\documentclass{amsart}
\usepackage{amssymb}
\usepackage{amsmath}
\usepackage{amsfonts}
\usepackage{mathrsfs}
\newtheorem{thm}{Theorem}[section]
\newtheorem{cor}[thm]{Corollary}
\newtheorem{lem}[thm]{Lemma}
\newtheorem{prop}[thm]{Proposition}
\theoremstyle{definition}

\theoremstyle{remark}

\numberwithin{equation}{section}
\usepackage{enumitem}
\usepackage{romannum}

\begin{document}
	
	%
	%
	%
	%
	%
	%
	%
	%
	%

	\title[Uniform Norms and Maximal Substructures]{On Multiplicity of Uniform Norms and Maximal Spectral Substructures in Commutative Banach Algebras}

	\author[J. J. Dabhi]{Jekwin J. Dabhi}
	\address{Institute of Infrastructure Technology Research and Management (IITRAM), Ahmedabad - 380026, Gujarat, India}
	\email{jekwin.13@gmail.com, jekwin.dabhi@iitram.ac.in}

	\author[P. A. Dabhi]{Prakash A. Dabhi}
	\address{Institute of Infrastructure Technology Research and Management (IITRAM), Ahmedabad - 380026, Gujarat, India}
	\email{lightatinfinite@gmail.com, prakashdabhi@iitram.ac.in}
	\thanks{The first author is thankful  for  research support from CSIR HRDG, India (file no. 09/1274(13914)/2022-EMR-I) for the Senior Research Fellowship. The second author is grateful to the National Board for Higher Mathematics (NBHM), India, for the research grant (02011/39/2025/NBHM(R. P.)/R \&D II/16090)}
	\subjclass{Primary 46J05, 46J40}
	
	\keywords{Commutative Banach algebra, Uniform norms, Spectral extension property, weak regularity.}
	
	\date{}
	
	\begin{abstract}
		Let $\mathcal A$ be a semisimple commutative Banach algebra. It is shown that either $\mathcal A$ has exactly one uniform norm or it admits uncountably many uniform norms. Further, it is shown that there always exists a largest closed subalgebra of $\mathcal A$ which is weakly regular, and that there always exist largest closed ideals in $\mathcal A$ having unique uniform norm property (UUNP) and spectral extension property (SEP) respectively.
	\end{abstract}

	\maketitle
	
	\section{Introduction}
	The algebraic and topological structures of commutative Banach algebras have been profoundly linked since the inception of Gel'fand theory. A cornerstone result in this domain is B. E. Johnson's Uniqueness of Norm Theorem \cite{J}, which establishes that any semisimple commutative Banach algebra admits a unique complete algebra norm topology. Given the uniqueness of complete norms, significant interest naturally shifted toward the study of incomplete submultiplicative norms \cite{M2, Y, G}. This investigation led to the introduction of the spectral extension property (SEP) by Meyer \cite{M1}, which characterizes when the spectral radius remains invariant under algebra extensions.
	
	The question of when a topological algebra admits a unique uniform norm was first investigated by Bhatt and Karia \cite{H4}. Motivated by this foundational work and the behavior of the SEP, Bhatt and Dedania \cite{H} formally introduced the unique uniform norm property (UUNP) for Banach algebras. A commutative Banach algebra $\mathcal A$ has the UUNP if it admits exactly one uniform norm, which must necessarily be the spectral radius. The study of the UUNP revealed a natural hierarchy of spectral properties: regularity implies weak regularity, which in turn implies the SEP, which finally implies the UUNP \cite{H, H3, M1}. These properties were subsequently studied on various specific algebras by authors in \cite{P3, P2}.
	
	If a semisimple commutative Banach algebra fails to possess the UUNP, it admits multiple uniform norms. A compelling question arises: how many distinct uniform norms can such an algebra admit? Bhatt and Dedania \cite{H2} first proved that if a Beurling algebra $L^1(G, \omega)$ admits at least two uniform norms, it must admit infinitely many. Later, the authors in \cite{P} generalized this phenomenon, proving that any commutative Banach algebra with more than one uniform norm must possess at least a countably infinite number of them. However, whether the cardinality of the set of uniform norms can be uncountable has remained an open question. In this paper, we answer this question affirmatively, proving that a semisimple commutative Banach algebra either possesses exactly one uniform norm or admits uncountably many.
	
	Parallel to the multiplicity of uniform norms, we investigate the existence of maximal algebraic substructures possessing these hierarchical spectral properties. In 1992, Inoue and Takahasi \cite{T} proved that if a Banach algebra is not regular, it always contains a largest closed regular subalgebra. Motivated by this phenomenon of maximal structures, we investigate whether analogous largest subalgebras and ideals exist for weaker properties. We establish that if a Banach algebra lacks weak regularity, UUNP, or SEP, it nevertheless always contains a largest closed subalgebra (or ideal) possessing these respective properties.
	
	Throughout the paper, $\mathcal A$ is assumed to be a commutative Banach algebra. A nonzero multiplicative linear functional on $\mathcal A$ is a \emph{complex homomorphism} on $\mathcal{A}$. Let $\Phi_{\mathcal A}$ be the collection of all complex homomorphisms on $\mathcal A$. The set $\Phi_{\mathcal A}$ is equipped with the weakest topology such that for each $x \in \mathcal A$, the mapping $\widehat x : \Phi_\mathcal A \rightarrow \mathbb C$ defined by $\widehat{x}(\varphi) = \varphi(x)$ is continuous. The set $\Phi_{\mathcal A}$ with this \emph{Gel'fand topology} is the \emph{Gel'fand space} of $\mathcal A$, which is a locally compact Hausdorff space. If $\mathcal A$ is unital, then $\Phi_{\mathcal A}$ is compact. A commutative Banach algebra $\mathcal A$ is \emph{semisimple} if $\cap_{\varphi \in \Phi_{\mathcal A}}\ker \varphi=\{0\}$, where $\ker \varphi=\{x \in \mathcal A:\varphi(x)=0\}$ is the kernel of $\varphi$. Let $I$ be a closed ideal in $\mathcal{A}$. Then the Gel'fand space of $I$ is identified with $\Phi_\mathcal{A} \setminus h(I)$, where $h(I) = \{\varphi \in \Phi_\mathcal{A} : I \subset \ker(\varphi) \}$ is the \emph{hull} of $I$.
	
	Let $X$ be a nonempty set. A subset $R \subset X$ is a \emph{boundary} for a collection $\mathscr{F}$ of bounded functions on $X$ if for every $f \in \mathscr{F}$ there exists a $y_f \in R$ such that $\|f\|_{\infty} = |f(y_f)|$ \cite[Def. 3.3.1]{K}. Let $X$ be a locally compact Hausdorff space and $\mathcal B$ a subalgebra of $C_0(X)$ that strongly separates the points of $X$. The intersection of all closed boundaries for $\mathcal{B}$, which is a boundary by \cite[Theorem 3.3.2]{K}, is the \emph{Shilov boundary} of $\mathcal {B}$, and is denoted by $\partial(\mathcal B)$. Let $\Gamma : \mathcal A \rightarrow C_0(\Phi_\mathcal A)$ defined by $\Gamma(x) = \widehat{x}$ be the Gel'fand representation of $\mathcal A$. A subset $R$ of $\Phi_\mathcal A$ is a \emph{boundary} for $\mathcal A$ if $R$ is a boundary for $\Gamma(\mathcal A)$. In particular, the \emph{Shilov boundary} of $\mathcal A$, denoted by $\partial(\mathcal A)$, is the intersection of all closed boundaries of $\Gamma(\mathcal{A})$.
	
	A submultiplicative norm (not necessarily complete) $|\cdot|$ on $\mathcal A$ is a \emph{uniform norm} if it satisfies the \emph{square property} $|x^2| = |x|^2$ for all $x \in \mathcal A$. If a commutative Banach algebra $\mathcal A$ is semisimple, it always admits a uniform norm, namely, the spectral radius $r(x)=\lim_{n\to \infty}\|x^n\|^{\frac{1}{n}}$ for all $x \in \mathcal A$. Note that $r(\cdot)$ can also be obtained by $r(x)=\sup\{|\widehat x(\varphi)|:\varphi \in \Phi_{\mathcal A}\}$ for all $x \in \mathcal A$. Conversely, if the spectral radius $r(\cdot)$ is a norm, then $\mathcal A$ is semisimple. For any uniform norm $|\cdot|$ on a semisimple commutative Banach algebra $\mathcal A$, we have $|\cdot| \leq r(\cdot)$, i.e., the spectral radius is the largest uniform norm on $\mathcal A$.
	
	For a uniform norm $|\cdot|$ on $\mathcal A$, the set $F = \{\varphi \in \Phi_\mathcal{A} : |\varphi(x)| \leq |x| \;(x \in \mathcal{A})\}$ is a closed subset of $\Phi_\mathcal A$. Let $|x|_F=\sup\{|\varphi(x)|:\varphi \in F\}$ for all $x \in \mathcal A$. Then $|\cdot|_F$ is a uniform norm on $\mathcal A$ and $|\cdot| =|\cdot|_F$ \cite[Lemma 4.6.2]{K}. In this case, $F$ is said to be a \emph{generating set} for the uniform norm $|\cdot|$. A closed subset $F$ of $\Phi_{\mathcal A}$ is a \emph{set of uniqueness} if $|\cdot|_F$ is a uniform norm on $\mathcal A$.
	
	Finally, we recall the definitions of the relevant spectral properties. A commutative semisimple Banach algebra $\mathcal{A}$ is \emph{regular} if it separates points from closed subsets of its Gel'fand space; specifically, for any closed $E \subset \Phi_{\mathcal{A}}$ and $\varphi \in \Phi_{\mathcal{A}} \setminus E$, there exists a nonzero $x \in \mathcal{A}$ such that $\widehat{x}(E) = \{0\}$ and $\widehat{x}(\varphi) \neq 0$ \cite[Def. 4.2.1]{K}. Weakenings of this property include \emph{weak regularity}, where for any nonempty proper closed subset $E \subset \Phi_{\mathcal{A}}$, there exists a nonzero $x \in \mathcal{A}$ such that $\widehat{x}(E) = \{0\}$.
	
	\section{multiplicity of uniform norms}
	\begin{prop}
		Let $\mathcal A$ be a commutative Banach algebra. Then either $\mathcal A$ has exactly one uniform norm or it has uncountably many uniform norms.
	\end{prop}
	\begin{proof}
		By \cite[Theorem 2.1(i)]{P}, any two distinct uniform norms on the unitization $\mathcal A_e$ of $\mathcal A$ corresponds to distinct uniform norms when restricted to $\mathcal A$. So, we may assume that $\mathcal A$ is unital.
		
		Assume that $\mathcal A$ admits at least two distinct uniform norms. Then $\mathcal A$ will have infinitely many uniform norms. So, we may choose uniform norms $p_1(\cdot)$ and $p_2(\cdot)$ such that $p_1(\cdot)\neq r(\cdot)$ and $p_2(\cdot)\neq r(\cdot)$. Let $K_1$ and $K_2$ be the corresponding generating closed sets for $p_1(\cdot)$ and $p_2(\cdot)$ respectively, i.e., $K_i=\{\varphi\in \Phi_{\mathcal A}:|\varphi(x)|\leq p_i(x)\;\;(x \in \mathcal A)\}$, $i=1,2$. As $\Phi_\mathcal A$ is compact, the sets $K_1$ and $K_2$ are compact. As $p_1(\cdot) \neq p_2(\cdot)$, there exists an element $y \in \mathcal A$ such that $p_1(y) < p_2(y)$. Consequently, there exists a character $\varphi_0 \in K_2$ such that $\sup_{\varphi \in K_1} |\widehat{y}(\varphi)| < |\widehat{y}(\varphi_0)|$. Take $x = \frac{y}{\widehat{y}(\varphi_0)}$. Then $\widehat{x}(\varphi_0) = 1$ and $p_1(x) = \sup_{\varphi \in K_1} |\widehat{x}(\varphi)| = r$, for some $0<r<1$.
		
		We claim that the spectrum of $x$, defined as $\sigma(x) = \widehat x (\Phi_\mathcal A)$, contains elements of every modulus $c$, where $r<c<1$. Suppose not, i.e., there exists some $r < c_0 < 1$ such that no element in $\sigma(x)$ has modulus $c_0$. Then this partitions compact spectrum $\sigma(x)$ into two disjoint open sets $$X_1 = \sigma(x) \cap \{z \in \mathbb C : |z| < c_0\} \ \text{and} \ X_2 = \sigma(x) \cap \{z \in \mathbb C : |z| > c_0\}.$$
		
		Note that if $\varphi \in K_1$, then $|\varphi(x)|\leq r$, i.e., $\varphi(x)\in X_1$. Also, as $\varphi_0(x)=1>c_0$, $\varphi_0(x)\in X_2$. As $\widehat x : \Phi_\mathcal A \rightarrow \mathbb C$ is continuous, $U_1 = \widehat x^{-1}(X_1)$ and $U_2 = \widehat x^{-1}(X_2)$ are two nonempty disjoint open subsets of $\Phi_{\mathcal A}$ whose union is $\Phi_\mathcal A$. By the Shilov Idempotent Theorem, there exists an idempotent $e \in \mathcal A$ such that $\widehat{e} (U_1) = \{0\}$ and $\widehat{e} (U_2) = \{1\}$. Let $\varphi \in K_1$. Then $|\widehat{x}(\varphi)| \leq r < c_0$ implies that $K_1 \subset U_1$, i.e., $\widehat{e}(K_1) = \{0\}$. Therefore $p_1(e) = \sup_{\varphi \in K_1} |\widehat{e}(\varphi)| = 0$. Now, $p_1(\cdot)$ being a norm forces $e$ to be $0$. But then $\widehat e(U_2)=\{1\}$ is not possible. This proves our claim, i.e.,  $\sigma(x)$ contains an element having modulus $c$ for all $r < c < 1$.
		
		Given $c\in (r,1)$, there exists $\varphi_c \in \Phi_\mathcal A$ such that $|\widehat{x}(\varphi_c)| = c$. Using this continuum of spectral values, we construct uncountably many distinct uniform norms. For  each $c \in [r,1]$, define the closed (and thus compact) subset $S_c \subset \Phi_\mathcal A$ by $$S_c = K_1 \cup \{ \varphi \in \Phi_\mathcal A:|\widehat{x}(\varphi)| \leq c \}$$ and $$p_c(a) =\sup\{|\varphi(a)|:\varphi\in S_c\}=\max \left( \sup_{\varphi \in K_1} |\widehat{a}(\varphi)|, \sup_{|\widehat{a}(\varphi)| \leq c} |\widehat{a}(\varphi)|   \right)\quad(a \in \mathcal A).$$
		Since $S_c$ contains $K_1$ and $K_1$ is a set of uniqueness,  $S_c$ will be a set of uniqueness, i.e., $p_c(\cdot)$ is a uniform norm. We established that $\sup_{\varphi \in K_1} |\widehat{x}(\varphi)| = r \leq c$, and because there exists $\varphi_c \in \Phi_\mathcal A$ with $|\widehat{x}(\varphi_c)| = c$, we obtain  $$p_c(x) = \max(r,c)=c.$$
		
		Here, for each $c,c' \in [r,1]$  the norms $p_c(\cdot)$ and $p_{c'}(\cdot)$  are not  equivalent uniform norms as equivalent uniform norms are equal and as for each $r<c<1$  each uniform norm $p_c(\cdot)$ admits different values for $x$. Since the interval $[r,1]$ is uncountable, it follows that $\mathcal A$ admits uncountably many uniform norms.
	\end{proof}

	\section{Largest subalgebras/ideals having spectral properties}
	Let $\mathcal{A}$ be a semisimple commutative Banach algebra. A norm $|\cdot|$ on $\mathcal A$ is a \emph{semisimple norm} if the completion of $\mathcal A$ in the norm $|\cdot|$ is semisimple \cite{Y}. Let $sN$ be the collection of all semisimple norms on $\mathcal A$. By \cite{M1}, the \emph{semisimple permanent radius} of $\mathcal A$ is defined as $r_s(x) = \inf\{|x| :|\cdot| \in sN\}$ and the  \emph{permanent radius}  is defined as $r_p(x) = \inf \{r_{\mathcal A_{p}}(x) : p \in N \}$, where $N$ is the collection of all submultiplicative norms on $\mathcal A$ and $r_{\mathcal A_{p}}(x)$ is the spectral radius of $x$ in the completion of $\mathcal A$ with respect to $p \in N$ \cite{Y}. A Banach algebra $\mathcal A$ has a $(P)$- property \cite{M1,M2} if for a given nonzero closed  ideal $I$ in $\mathcal{A}$, there  exists a nonzero $x \in I$ such that $r_p(x) > 0$.
	
	We use the following results.
	
	\begin{thm}\cite[Theorem 2.3]{H}\label{SH1}
		Let $\mathcal A$ be a semisimple commutative Banach algebra. Then the following statements are equivalent.
		\begin{enumerate}
			\item $\mathcal A$ has UUNP.
			\item $\partial{\mathcal A}$ is the smallest closed set of uniqueness.
			\item If $F\subset \Phi_{\mathcal A}$ is closed and not containing $\partial {\mathcal A}$, then there exists $a\in \mathcal A$ such that $r(a)>0$ and $\widehat a|F=0$.
			\item If $F\subset \Phi_{\mathcal A}$ is closed and does not contain $\partial{\mathcal A}$, then there exists $a\in \mathcal A$ such that $r_s(a)>0$ and $\widehat a|F=0$.
			\item $r=r_s$ on $\mathcal A$.
		\end{enumerate}
	\end{thm}
	
	\begin{thm}\cite[Proposition 2.4]{H}\label{SH2}
		Let $\mathcal A$ be a semisimple commutative Banach algebra. Then the following statements are equivalent.
		\begin{enumerate}
			\item $\mathcal A$ has SEP.
			\item $\mathcal A$ has UUNP and $r_s=r_p$, i.e., $\mathcal A$ has UUNP and every norm dominates a semisimple norm.
			\item $\mathcal A$ has UUNP and $(P)$- property.
		\end{enumerate}
	\end{thm}
	
	Now, we move towards our results.
	
	\begin{lem} \label{3.1} Let $\mathcal A$ be a commutative  Banach algebra, and let $E$ be the closed subalgebra of $\mathcal A$ generated by the collection $\{E_i\}_{i \in I}$ of closed subalgebras of $\mathcal A$ such that each $E_i$ has UUNP. Then  $\bigcup_{i \in I} \partial E_{i} \subset \partial E$.
	\end{lem}
	\begin{proof}
		We may assume that $\Phi_\mathcal{A}$ is compact by \cite[Theorem 3.1]{H1}. For $j \in I$ define $\Psi_{j} : \Phi_E \cup \{0\} \rightarrow{\Phi_{E_{j}}} \cup \{0\}$ by $\Psi_{j}(\varphi) = \varphi_{|_{E_{j}}}$. Notice that  $\Psi_{j}$ is   $w^\ast$- continuous and $\Phi_\mathcal  A \cup \{0\}$ is a $w^\ast$- compact subset of dual of $\mathcal  A$ \cite[Section 2.2]{K}. Since $\partial E \cup \{0\}$ is a closed subset of the compact set $\Phi_E \cup \{0\}$, the set $\partial E \cup \{0\}$ is compact and  the image $\Psi_{j}(\partial E \cup \{0\})$ is a closed subset of $\Phi_{E_{j}} \cup \{0\}$. Let $F = \Psi_{j}(\partial E \cup \{0\}) \cap \Phi_{E_{j}}$. To show that $F$ is a closed subset of $\Phi_{E_{j}}$, let $\varphi \in \Phi_{E_{j}}$ be a limit point of $F$. Then $\varphi$ is also a limit point of $\Psi_{j}(\partial E \cup \{0\})$. Since $\Psi_{j}(\partial E \cup \{0\})$ is closed in $\Phi_{E_{j}} \cup \{0\}$, we have $\varphi \in \Psi_{j}(\partial E \cup \{0\})$. Therefore, $\varphi \in F$, which implies that $F$ is closed. For $x \in E_{j}$, define
		\begin{align*}
			|x| &= \sup\{ |\varphi(x)| : \varphi \in F \} \\
			&=  \sup\{ |\varphi(x)| : \varphi \in \Psi_{j}(\partial E \cup \{0\} ) \setminus \{0\}\}.\
		\end{align*}
		Then $|\cdot|$ is a norm on $E_{j}$. Consequently, $F$ is a closed set of uniqueness for $E_{j}$, which, by Theorem \ref{SH1}, yields $\partial E_{j} \subset F \subset \Psi_{j}(\partial E) \setminus \{0\}$. This implies $\Psi_{j}^{-1}(\partial E_{j}) \subset \partial E$. As $j \in I$ is arbitrary, we get $\bigcup_{i \in I} \partial E_i \subset \partial E$.
	\end{proof}
	
	\begin{thm}\label{4.2}
		Let $\mathcal A$ be a commutative Banach algebra. Then there always exists largest closed subalgebra of $\mathcal A$ which is weakly regular.
	\end{thm}
	\begin{proof}
		Let $(E_i)_{i \in I}$ be the collection of all weakly regular closed subalgebras of $\mathcal{A}$. Since weak regularity implies UUNP, each $E_i$ possesses the UUNP. Let $E$ be the closed subalgebra of $\mathcal A$ generated by $\bigcup_{i \in I} E_i$. Then, by Lemma \ref{3.1}, we have $\bigcup_{i \in I} \partial E_i \subset \partial E$. Now, let $\varphi \in \Phi_E$. Then $\varphi |_{E_{i'}} \neq 0$ for some $i' \in I$, and hence $\Phi_E \subset \bigcup_{i \in I} \Phi_{E_i}$. Because each $E_i$ is weakly regular, $\partial E_i = \Phi_{E_i}$ for all $i \in I$. Consequently, we obtain the chain of inclusions $\bigcup_{i \in I} \partial E_i \subset \partial E \subset \Phi_E \subset \bigcup_{i \in I} \partial E_i$, which implies that $\partial E = \Phi_E$.
		
		Now, our claim is to show that $E$ has UUNP. For that let $F$ be a closed subset of $\Phi_E$ not containing $\partial E$. Then $F$ does not contain $\partial E_{i_{0}}$ for some $i_{0}$.  Define $\Psi_{i_0}$ as in Lemma \ref{3.1}. Then $\Psi_{i_0}(F)$ is closed, it is disjoint from $\partial E_{i_{0}}$, and hence, by Theorem \ref{SH1}, there exists a nonzero $x \in E_{i_{0}}$ such that $\widehat{x}\{\Psi_{i_0}(F)\} = \widehat{x}(F) = \{0\}$ and $r_{s}(x) > 0$. As $r_{E}(x) \geq r_{s}(x) > 0$, by  Theorem \ref{SH1}, $E$ has UUNP. As $E$ has UUNP and $\Phi_E = \partial E$, it follows, by Theorem \ref{SH1}, that $E$ is weakly regular.
	\end{proof}
	
	\begin{cor}
		Let $\mathcal A$ be a commutative Banach algebra. Then there exists a largest closed ideal in $\mathcal A$ which is weakly regular.
	\end{cor}
	
	\begin{proof}
		Let $(J_i)_{i \in I}$ be the collection of all closed ideals in $\mathcal A$ which are weakly regular, and let $J$ be the closed ideal generated by $\bigcup_i J_i$. Then $J$ is weakly regular.	The proof is similar to that of Lemma \ref{3.1} and Theorem \ref{4.2}.
	\end{proof}
	
	\begin{thm}\label{3.4}
		Let $\mathcal A$ be a commutative Banach algebra and let $(J_i)_{i \in I}$ be the collection of all closed ideals in $\mathcal A$ each having UUNP. Let $J$ be the closed ideal generated by $\bigcup_i J_i$. Then $\overline{\bigcup_{i \in I}\partial J_i} = \partial J$.
	\end{thm}
	\begin{proof}
		Note that the Gel'fand space of  $J$ is homeomorphic to $\bigcup_{i \in I} \Phi_{J_i}$, where $\Phi_{J_i} = \Phi_{\mathcal A} \setminus h(J_i)$. Now, our claim is to show that $\overline{\bigcup_{i \in I} \partial J_i} =  \partial J$. By Lemma \ref{3.1}, we have $\overline{\bigcup_{i \in I} \partial J_i} \subset \partial J$. For the reverse inclusion, we prove that the image of $\partial J$ under the mapping $\Psi_i$, as defined in Lemma \ref{3.1}, is equal to $\partial J_i$ for each $i \in I$. Suppose to the contradiction there exists $i_0 \in I$ such that $\Psi_{i_0}(\partial J) \nsubseteq \partial J_{i_0}$.  Then there exists $\varphi_0 \in \Psi_{i_0}(\partial J)$ such that $\varphi_0 \notin \partial J_{i_0}$. As $\partial J_{i_0}$ is closed in $\Phi_{J_{i_0}}$, there exists an open subset $U$ of $\Phi_{J_{i_0}}$ such that $\varphi_0 \in U$ and $U \cap \partial J_{i_0} = \emptyset$. Now, let $V = \Psi_{i_0}^{-1}(U)$. Then $V$ is open in $\Phi_J$ and contains $\varphi_0 \in \partial J$. Then, by \cite[Corollary 3.3.4]{K}, there exists some $x \in J$ such that $$\|x|_{\Phi_J \setminus V}\|_{\infty} < \|x|_{V}\|_{\infty}.$$ By multiplying with suitable factor we may assume that $\|x|_{V}\|_{\infty} = 1$. For $\varphi_0 \in \Phi_{J_{i_0}}$ choose $y \in J$ such that $\varphi_0(y) = 1$. Define $z_n = x^ny$. As $J_{i_0}$ is an ideal, $z_n \in J_{i_0}$ for each $n$. For $K = \|y|_{{\Phi_J \setminus V}}\|_\infty $, choose large enough $n_0$ such that $\|x^{n_0}|_{\Phi_J \setminus V}\|_{\infty} < \frac{1}{K+1}$. Now,
		\begin{align*}
			\|z_{n_0}|_{\Phi_{J} \setminus V}\|_{\infty} &= \|x^{n_0}y|_{\Phi_{J} \setminus V}\|_{\infty} \\
			&\leq K\|x^{n_0}|_{\Phi_J \setminus V}\|_{\infty} \\
			&< 1.
		\end{align*}
		Note that $\|z_{n_0}|_{\Phi_{J_{i_0}} \setminus U}\|_{\infty} \leq \|z_{n_0}|_{\Phi_{J} \setminus V}\|_{\infty} < 1$ and $|\varphi_0(z_{n_0})| = 1$. This implies that $\|z_{n_0}|_{\Phi_{J_{i_0}} \setminus U}\|_{\infty} < \|z_{n_0}|_{U}\|_{\infty}$, that is, $|\widehat{z_{n_0}}|$ attains its maximum outside $\partial J_{i_0}$. As  $\partial J_{i_0} \cap U = \emptyset$, this is not possible.  Hence our assumption is wrong. Therefore, we have $\Psi_{i}(\partial J) = \partial{J_{i}}$ for all $i \in I$ which implies that $\partial J \subset \bigcup_{i \in I} \partial J_{i}$. Hence the proof.
	\end{proof}
	
	\begin{prop}\label{3.4.1}
		Let $\mathcal{A}$ be a commutative Banach algebra. Then there always exists largest closed ideal having UUNP.
	\end{prop}
	\begin{proof}
		Let $(J_i)_{i \in I}$ be the collection of all closed ideals in $\mathcal{A}$ having the UUNP, and let  $J$ be the closed ideal in $\mathcal A$ generated by $\bigcup_i J_i$. By Theorem \ref{3.4}, we have $\overline{\bigcup_{i \in I} \partial J_{i}} = \partial J$. Let $F$ be a closed subset of $\Phi_J$ not containing $\partial J$. Then $F$ does not contain $\partial E_{i_0}$ for some $i_0 \in I$. Define $\Psi_{i_0}$ as in Lemma \ref{3.1}. Thus $\Psi_{i_0}(F)$ is a closed set disjoint from $\partial J_{i_0}$, and hence, by Theorem \ref{SH1}, there exists a nonzero $x \in J_{i_0}$ such that $\widehat{x}(\Psi_{i_0}(F)) = \widehat{x}(F) = \{0\}$ and $r_s(x) > 0$. Because $r_J(x) \geq r_s(x) > 0$, it follows that $J$ possesses UUNP.
	\end{proof}
	
	\begin{thm}
		Let $\mathcal{A}$ be a commutative Banach algebra. Then there always exists largest closed ideal having SEP.
	\end{thm}
	
	\begin{proof}
		Let $(J_i)_{i \in I}$ be the collection of all closed ideals in $\mathcal A$ having SEP and let  $J$ be the closed ideal in $\mathcal A$ generated by $\bigcup_i J_i$. Then $J_i$ has UUNP for each $i \in I$, by Theorem \ref{SH1}. Hence $J$ has UUNP, by Proposition \ref{3.4.1}. Now, by Theorem \ref{SH2}, it is sufficient to show that $J$ has a $(P)$-property, that is, for  a nonzero closed ideal $I$ of $J$ there exists $x \in I$ such that $r_1(x) = \inf\{|x| : |\cdot| \ \text{is a norm on} \ J\} > 0$. Let $I$ be a nonzero closed ideal in $J$. Then $I \cap J_{i_0}$ is nonzero for some $i_0$. As $I \cap J_{i_0}$ is a closed ideal in $J_{i_0}$, there exists $x_0 \in J_{i_0}$ such that $r_2(x_0) = \inf\{|x_0| : |\cdot| \ \text{is a norm on}\ J_{i_0}\} > 0 $. Here, $r_1(x_0) \geq r_2(x_0) > 0$. Therefore $J$ has SEP. This proves the result.
	\end{proof}

\end{document}